\documentclass[12pt]{amsart}
\usepackage[a4paper]{geometry}
\usepackage{color}
\usepackage{comment}
 \newtheorem{thm}{Theorem}[section]
 \newtheorem{cor}[thm]{Corollary}
 \newtheorem{lem}[thm]{Lemma}
 
 \theoremstyle{definition}
 \newtheorem{defn}{Definition}[section]
 \theoremstyle{remark}
 
 \numberwithin{equation}{section}
\newcommand{\grad}{\mathop{\mathrm{grad}}}
\newcommand{\trace}{\mathop{\mathrm{trace}}}
\newcommand{\dive}{\mathop{\mathrm{div}}}

\begin{document}
\title{\textbf{Biconservative hypersurfaces in space forms $\overline{M}^{\lowercase{n+1}}(\lowercase{c})$}}
\author{Ram Shankar Gupta and Andreas Arvanitoyeorgos}
 \maketitle
% -----------------------------------------------------------
\begin{abstract}
  In this paper we study biconservative hypersurfaces $M$
in space forms $\overline M^{n+1}(c)$ with four distinct principal
curvatures whose second fundamental form has constant norm. We prove
that every such  hypersurface has constant mean
curvature and constant scalar curvature. 
\end{abstract}

%%% ----------------------------------------------------------------------
\maketitle
%%% ----------------------------------------------------------------------
\section{\textbf{Introduction}}

In the last three decades, one of the interesting research topics in
Differential Geometry is the study of biharmonic maps, in
particular, biharmonic immersions, between Riemannian manifolds. A
generalization of harmonic maps, proposed in 1964 by Eells and
Sampson \cite{5}, these maps are critical points of the bienergy
functional, obtained by integrating the squared norm of the tension
field, and are characterized by the vanishing of the bitension
field. Another interesting research direction, derived from here, is
the study of {\it biconservative} submanifolds, i.e., those submanifolds
for which only the tangent part of the bitension field vanishes.

In 1924, Hilbert pointed that the stress-energy tensor associated to
a functional $E$, is a conservative symmetric 2-covariant tensor $S$
at the critical points of $E$, i.e. $\dive S = 0$ (\cite{3}). For the
bienergy functional $E_2$, Jiang defined the stress-bienergy tensor
$S_2$ and proved that it satisfies $\dive S_2 =
-\langle\tau_2(\phi), d\phi\rangle$ (\cite{4}). Thus, if $\phi$ is
biharmonic, then $\dive S_2 = 0$. For biharmonic submanifolds, from
the above relation, we see that $\dive S_2 = 0$ if and only if the
tangent part of the bitension field vanishes. In particular, an
isometric immersion $\phi: (M, g)\rightarrow(N, h)$ is called
biconservative if $\dive S_2 = 0$.

 In a different setting, B. Y.
Chen defined {\it biharmonic} submanifolds $M$ of the Euclidean space as
those with harmonic mean curvature vector field, that is $\Delta \vec{H}
=0$, where $\Delta$ is the  Laplacian operator. If we apply the
definition of biharmonic maps to Riemannian immersions into the
Euclidean space, we recover Chen's notion of biharmonic submanifolds.
Thus biharmonic Riemannian immersions can also be thought of as a
generalization of Chen's biharmonic submanifolds. The biharmonic
submanifolds were studied in
\cite{gu4,biha7,r20,r23,r13,biha15,r14,r18,12} and references therein.
\let\thefootnote\relax\footnote{\textbf{Acknowledgments:}
This work is supported by award of grant under FRGS for the year
2017-18, F.No. GGSIPU/DRC/Ph.D./Adm./2017/493.  The authors would like to thank the referee for useful suggestions on the paper.}

The biconservative submanifolds were studied and classified in
$\mathbb{E}^3$ and $\mathbb{E}^4$ by Hasanis and Vlachos (\cite{6}), in
which the biconservative hypersurfaces were called $H$-hypersurfaces.
The terminology ``biconservative" was first introduced in \cite{cadmo}.
The classification of $H$-hypersurfaces with three distinct curvatures
in Euclidean space of arbitrary dimension were obtained by Turgay in
\cite{k2}.  The classification of biconservative hypersurfaces in
 $\mathbb{E}^5_2$ with diagonal shape operator having three distinct principal
 curvatures was obtained by  Upadhyay and  Turgay in \cite{13}. Also, the first author and Sharfuddin  proved that every biconservative Lorentz hypersurface in
 $\mathbb{E}^{n+1}_1$ with complex eigenvalues has
constant mean curvature (\cite{gu1}). For more work on biconservative
hypersurfaces in pseudo-Euclidean spaces, please see (\cite{13,gu1})
and references therein.

 The constant mean curvature (CMC) biconservative surfaces
in $\mathbb{S}^n \times \mathbb{R}$ and $\mathbb{H}^n \times
\mathbb{R}$ were studied in \cite{7} by Fetcu et al. A complete classification of
CMC biconservative surfaces in a four-dimensional space form was
given in \cite{gu3} by Montaldo et al. Further, the classification of biconservative
hypersurfaces in $\mathbb{S}^4$ and $\mathbb{H}^4$ was obtained in
\cite{gu2} by Turgay and Upadhyay. 
A survey about biharmonic and biconservative hypersurfaces in space forms is provided in \cite{FeOn} by  Fetcu et al. 
In  \cite{gu5} the first author studied the biconservative hypersurfaces in Euclidean 5-space with constant norm
of second fundamental form and one of the results states that every such hypersurface has constant mean curvature. Finally, some recent results on 
complete biconservative surfaces and on closed biconservative surfaces in space forms was obtained by Nistor and Oniciuc (\cite{NiCe})  and Montaldo and by Pampano (\cite{MoPa}). 

\smallskip
 In view of the above developments, in the present paper we study biconservative
hypersurfaces in a space form $\overline{M}^{n+1}(c)$  with at most four
distinct principal curvatures, whose second fundamental form has
constant norm. The main result is  the following:

\begin{thm}
Every biconservative  hypersurface $M$ in a space form
$\overline{M}^{n+1}(c)$ with at most four distinct principal curvatures,
whose second fundamental form has constant norm, is of constant mean
curvature and of constant scalar curvature.
 \end{thm}

\smallskip
 Also, the class of biconservative hypersurface in space forms contains some  other class of  hypersurfaces, that is, the  equation (\ref{e:b9}) is fulfilled by a biharmonic submanifold $\triangle \vec
 {H}=0$ and a hypersurface with $\triangle \vec
 {H}=\lambda \vec{H}$ \cite{biha7}. 
 
   Therefore,  from Theorem 1.1, we have following:
\begin{cor}
Every biharmonic  hypersurface $M$ in a space form
$\overline{M}^{n+1}(c)$ with at most four distinct principal curvatures,
whose second fundamental form has constant norm, is of constant mean
curvature and of constant scalar curvature. \end{cor}
 \begin{cor}
Every hypersurface $M$ in a space form
$\overline{M}^{n+1}(c)$ satisfying $\triangle \vec
 {H}=\lambda \vec{H}$ with at most four distinct principal curvatures,
whose second fundamental form has constant norm, is of constant mean
curvature and of constant scalar curvature.
 \end{cor}

 The paper is organized as follows:  In Section 2 we give some preliminaries.  In Section 3 we discuss biconservative hypersurfaces in space forms with four distinct principal curvatures.  We obtain some condition for eigenvalues (Lemma 3.1) and expressions of covariant derivatives of an orthonormal frame in terms of connection forms (Lemma 3.2).  In Section 4 we use the condition that the second fundamental form has constant norm, and obtain further simplifications of the connection forms (Lemmas 4.1 and 4.2).  Section 5 is devoted to the proof of Theorem 1.1.  In the process, a technical tool about resultant of polynomials is extensively used (Lemma 2.1).
 
 Finally, we would like to note that Lemma \ref{l:3}  is crucial to obtain our results in the case of at most four distinct principal curvatures. In case of more than four distinct principal curvatures, it seems difficult to obtain an analogue of this lemma.  We may discuss this in a separate paper.

\section{\textbf{Preliminaries}}

  Let ($M, g$) be a hypersurface isometrically immersed in a $(n+1)$-dimensional space form
$(\overline{M}^{n+1}(c), \overline g)$ and $g = \overline g_{|M}$.
We denote by $\xi$ a unit normal vector to $M$ where
$\overline{g}(\xi, \xi)= 1$, and by $h$ and $ \mathcal{A}$ the second fundamental form and the shape operator of $M$, where $h(X, Y)=g(\mathcal{A}(X), Y)$, for $ X, Y \in \Gamma(TM)$.

 The mean curvature $H$ of $M$ is given by
\begin{equation}\label{e:b4}
H = \frac{1}{n} \trace \mathcal{A}.
\end{equation}

 Let $\nabla$ denotes the Levi-Civita connection on $M$. Then, the Gauss and Codazzi equations are given by
\begin{equation}\label{e:b5}
R(X, Y)Z = c(g(Y, Z) X - g(X, Z) Y)+g(\mathcal{A}Y, Z) \mathcal{A}X - g(\mathcal{A}X, Z) \mathcal{A}Y,
\end{equation}
\begin{equation}\label{e:b6}
(\nabla_{X}\mathcal{A})Y = (\nabla_{Y}\mathcal{A})X,
\end{equation}
respectively, where $R$ is the curvature tensor and
\begin{equation}\label{e:b7}
(\nabla_{X}\mathcal{A})Y = \nabla_{X}\mathcal{A} Y-
\mathcal{A}(\nabla_{X}Y)
\end{equation}
for all $ X, Y, Z \in \Gamma(TM)$.

 A submanifold satisfying $\triangle \vec {H} = 0$,
is called {\it biharmonic submanifold} (\cite{biha7}).
The biharmonic
equation  for $M^n$  in  a $(n+1)$-dimensional space form $\overline
M^{n+1}(c)$  can be decomposed into its normal and tangent part (\cite{biha7}, \cite{C-On}),  that is
\begin{eqnarray}
2\mathcal{A} (\grad H)+ n  H \grad H = 0,\label{e:b9}\\
 \Delta H+H({\rm trace}(\mathcal{A}^2)-n c)=0.
\end{eqnarray}
\begin{defn}
A submanifold satisfying equation (\ref{e:b9})
 is called biconservative.
\end{defn}
In the present work we are concerned with biconservative
hypersurfaces $M^n$  in  a $(n+1)$-dimensional space form $\overline
M^{n+1}(c)$.

%\begin{equation}\label{e:b9}
%2\mathcal{A} (\grad H)+ n  H \grad H = 0.
%\end{equation}

\smallskip
The following algebraic lemma will be useful to get our result:

\begin{lem} {\rm (\cite[Theorem 4.4, pp. 58--59]{ke})} Let D be a unique factorization domain, and let $f(X) =
a_{0}X^{m} +a_{1}X^{m-1} + \dots + a_{m}, g(X) = b_{0}X^{n} +
b_{1}X^{n-1} + \dots + b_{n}$ be two polynomials in $D[X]$. Assume
that the leading coefficients $a_{0}$ and $b_{0}$ of $f(X)$ and
$g(X)$ are not both zero. Then $f(X)$ and $g(X)$ have a nonconstant
common factor iff the resultant $\Re(f, g)$ of $f$ and $g$ is zero,
where
\begin{center}
$\Re(f,g)=
\begin{vmatrix}
  a_{0} & a_{1} & a_{2} & \cdots & a_{m} &   &   &   \\
    & a_{0} & a_{1} & \cdots & \cdots & a_{m} &   &   \\
    &   & \ddots & \ddots & \ddots & \ddots & \ddots &   \\
   &   &   & a_{0} & a_{1} & a_{2} & \cdots & a_{m} \\
  b_{0} & b_{1} & b_{2} & \cdots & b_{n} &   &   &   \\
    & b_{0} & b_{1} & \cdots & \cdots & b_{n} &  &   \\
    &   & \ddots & \ddots & \ddots & \ddots & \ddots &   \\
    &   &   & b_{0} & b_{1} & b_{2} & \cdots & b_{n} \\
\end{vmatrix},
$
\end{center}
and there are n rows of ``a" entries and m rows of ``b" entries.
\end{lem}

\section{\textbf{Biconservative hypersurfaces in
 space forms with four distinct principal curvatures}}

In this section, we study biconservative hypersurfaces with four
distinct principal curvatures in space forms. In view of
(\ref{e:b9}), it is easy to see that any CMC hypersurface is
biconservative. Therefore, we are interested in the study of
non CMC biconservative hypersurfaces in a space form $\overline
M^{n+1}(c)$.

We assume that the mean curvature is not
constant, and we will end up to a contradiction.
This implies that $\grad H\neq0$, hence there exists an
 open connected subset $U$ of $M^n$ with
grad$_{x}H \neq 0$, for all $x\in U$. From (\ref{e:b9}), it is easy
to see that $\grad H$ is an eigenvector of the shape operator
$\mathcal{A}$ with the corresponding principal curvature
$-\frac{nH}{2}$.

We denote by $A, B$, the following sets
$$ A=\{1, 2,\dots,n\},\quad  B=\{2, 3,\dots, n\}.$$

Without losing generality, we choose $e_{1}$ in the direction of
$\grad H$. Then, the shape operator $\mathcal{A}$ of a hypersurface
$M^n$ in a space form $\overline M^{n+1}(c)$ takes the following
form with respect to a suitable orthonormal frame $\{e_{1},
e_{2},\dots, e_{n}\}$,
\begin{equation}\label{e:c1} \mathcal{A}e_i=\lambda_ie_i, \quad i \in
A,
\end{equation}
where $\lambda_i$ is the eigenvalue corresponding to the eigenvector $e_i$
of the shape operator.

 Then $\grad H$ can be expressed as
\begin{equation}\label{e:c2}
\grad H =\sum_{i=1}^{n} e_{i}(H)e_{i}.
\end{equation}

 As we have taken $e_{1}$ parallel to $\grad H$, it follows that
\begin{equation}\label{e:c3}
e_{1}(H)\neq 0, \quad e_{i}(H)= 0, \quad i \in B.
\end{equation}

  We express
\begin{equation}\label{e:c4}
\nabla_{e_{i}}e_{j}=\sum_{m=1}^{n}\omega_{ij}^{m}e_{m}, \quad i, j
\in A.
\end{equation}

From (\ref{e:c4}) and using the compatibility conditions
$(\nabla_{e_{k}}g)(e_{i}, e_{i})= 0$ and $(\nabla_{e_{k}}g)(e_{i},
e_{j})= 0$, we obtain
\begin{equation}\label{e:c5}
\omega_{ki}^{i}=0, \hspace{1 cm} \omega_{ki}^{j}+ \omega_{kj}^{i}
=0,
\end{equation}
for $i \neq j, $ and $i, j, k \in A$.

Taking $X=e_{i}, Y=e_{j}$ in (\ref{e:b7}) and using (\ref{e:c1}),
(\ref{e:c4}), we get
$$(\nabla_{e_{i}}\mathcal{A})e_{j}=e_{i}(\lambda_{j})e_{j}+\sum_{k=1}^{n}\omega_{ij}^{k}e_{k}
(\lambda_{j}-\lambda_{k}).$$

Putting the value of $(\nabla_{e_{i}}\mathcal{A})e_{j}$ in
(\ref{e:b6}), we find
$$e_{i}(\lambda_{j})e_{j}+\sum_{k=1}^{n}\omega_{ij}^{k}e_{k}
(\lambda_{j}-\lambda_{k})=e_{j}(\lambda_{i})e_{i}+\sum_{k=1}^{n}\omega_{ji}^{k}e_{k}
(\lambda_{i}-\lambda_{k}).$$

Using $i\neq j=k$ and $i\neq j \neq k$ in the above equation, we
obtain
\begin{equation}\label{e:c6}
e_{i}(\lambda_{j})=
(\lambda_{i}-\lambda_{j})\omega_{ji}^{j}=(\lambda_{j}-\lambda_{i})\omega_{jj}^{i},
\end{equation}
and
\begin{equation}\label{e:c7}
(\lambda_{i}-\lambda_{j})\omega_{ki}^{j}=
(\lambda_{k}-\lambda_{j})\omega_{ik}^{j},
\end{equation}
respectively,  for $i, j, k \in A$.

Now, we have

\begin{lem}\label{l:1}
 Let $M^n$ be a biconservative hypersurface  in
  a space form $\overline M^{n+1}(c)$ having the shape operator given by
$(\ref{e:c1})$ with respect to a suitable orthonormal frame
$\{e_{1}, e_{2},\dots, e_{n}\}$. Then,
\begin{equation}\label{e:k4}
\lambda_1\neq\lambda_j,\ \ \mbox{for all}
\ \ j\in B.
\end{equation}
\end{lem}

\begin{proof} Assume that  $\lambda_{j}= \lambda_{1}$ for some $j\neq 1$. Then, from (\ref{e:c6}) we get that
$$e_1(\lambda_j)=0, \quad\mbox{or}\quad e_1(H)=0, \quad\mbox{since}\quad
\lambda_1=-\frac{nH}{2},$$ which contradicts (\ref{e:c3}), and this completes the proof.
\end{proof}

 Therefore, in view of the Lemma 3.1,
$\lambda_1= -\frac{n}{2} H$ has multiplicity one.  Since $M$ has
four distinct principal curvatures, we can assume that
$\lambda_1$, $\lambda_u, \lambda_v$ and $\lambda_w$
are the four distinct principal curvatures of the hypersurface $M$ with
multiplicities $1$, $p, q$ and $r$ respectively, such that
$$\lambda_2=\lambda_3=\dots=\lambda_{p+1}=\lambda_u, \ u\in C_1,$$
$$\lambda_{p+2}=\lambda_{p+3}=\dots=\lambda_{p+q+1}=\lambda_v, \ v \in C_2,$$
$$\lambda_{p+q+2}=\lambda_{p+q+3}=\dots=\lambda_{p+q+r+1}=\lambda_w,\ w\in C_3.$$
Here $p+q+r+1=n$ and $C_1, C_2$ and $C_3$ denote the sets
$$C_1=\{2, 3,\dots, p+1\}, \quad C_2=\{p+2, p+3,\dots, p+q+1\},\quad C_3=\{p+q+2, p+q+3,\dots, n\}.$$

 Using (\ref{e:b4}) and (\ref{e:c1}), we obtain that
 \begin{equation}\label{e:c11}
\sum_{j=2}^n\lambda_j= p\lambda_u+
q\lambda_v+r\lambda_w=\frac{3n}{2}H=-3\lambda_1,\ u\in C_1, \ v\in C_2, \ w\in C_3.
\end{equation}

Next, we have

\begin{lem}\label{l:2}
 Let $M^n$ be a biconservative hypersurface with four distinct principal curvatures in
  a space form $\overline M^{n+1}(c)$ having the shape operator given by
$(\ref{e:c1})$ with respect to a suitable orthonormal frame
$\{e_{1}, e_{2},\dots, e_{n}\}$. Then the following relations are satisfied:
\begin{eqnarray} \nabla_{e_{1}}e_{1}=0, \quad \nabla_{e_{1}}e_{i}=\sum_{ C_{s}}\omega_{1i}^me_{m}\ \ \mbox{for all}\ \  i \in
C_{s}, \  s\in\{1,2,3\}, \quad m\neq i,\nonumber\end{eqnarray}
 \begin{eqnarray}
\nabla_{e_{i}}e_{1}= -\omega_{ii}^{1}e_{i} \ \ \mbox{for all}\ \  i \in B, \quad \nabla_{e_{i}}e_{i}=\sum_{m=1}^n\omega_{ii}^{m}e_{m}
\ \ \mbox{for all}\ \  i \in B, \quad m\neq i,\nonumber \end{eqnarray}
\begin{eqnarray}\nabla_{e_{i}}e_{j}=\sum_{ C_s}\omega_{ij}^{m}e_{m} \ \ \mbox{for all}\ \  i, j \in C_s, \ s\in\{1,2,3\},\quad i\neq j, \quad j\neq m, \nonumber \end{eqnarray}
\begin{eqnarray}
\nabla_{e_{i}}e_{j}=\omega_{ij}^i e_i+\sum_{B\setminus C_s}\omega_{ij}^{m}e_{m} \ \ \mbox{for all}\ \  i \in C_s, \quad j \in
B\setminus C_s, \ s\in\{1,2,3\}, \quad m\neq j,\nonumber \end{eqnarray}
where $\sum_{C_s}$ and $\sum_{B\setminus C_s}$ denote
the summation taken over the corresponding $C_s$ and $B\setminus C_s $, respectively  for  $s\in \{1, 2,
3\},$ and $\omega_{ij}^{i}$ satisfy $(\ref{e:c5})$ and $(\ref{e:c6})$.\\
\end{lem}
\begin{proof} Using (\ref{e:c3}), (\ref{e:c4}) and the fact
that $[e_{i} \hspace{.1 cm}
e_{j}](H)=0=\nabla_{e_{i}}e_{j}(H)-\nabla_{e_{j}}e_{i}(H)=\omega_{ij}^{1}e_{1}(H)-\omega_{ji}^{1}e_{1}(H),$
for $i\neq j$, we find that
\begin{equation}\label{e:c9}
\omega_{ij}^{1}=\omega_{ji}^{1},  \quad i, j \in B.
\end{equation}

 Putting $i\neq 1, j = 1$ in (\ref{e:c6}) and using (\ref{e:c5}) and (\ref{e:c3}), we
 find
\begin{equation}\label{e:c12}
\omega_{1i}^{1}= 0, \quad  i\in A.
\end{equation}

 Putting $i = 1$  in (\ref{e:c7}), we
 obtain
 \begin{equation}\begin{array}{lcl}\label{e:b17}
\omega_{k1}^{j}= 0, \quad j\neq k \quad  \mbox{and}\quad j, k \in C_{s}, \ s\in \{1, 2, 3\}.
\end{array}\end{equation}

 Taking $i \in C_s, \ s\in \{1, 2, 3\}$ in (\ref{e:c7}), we
 have
\begin{equation}\begin{array}{lcl}\label{e:b18}
\omega_{ki}^{j}= 0, \quad j\neq k\quad\mbox{and}\quad j, k \in
B\setminus C_s.
\end{array}\end{equation}

 Putting  $j = 1$ in (\ref{e:c7}) and using (\ref{e:c9}), we
 get
\begin{equation}\begin{array}{lcl}\label{e:b19}
\omega_{ki}^{1}=\omega_{ik}^{1}=0, \quad i \in C_s,\quad k \in
B\setminus C_s,  \ s\in \{ 1, 2, 3\}.
\end{array}\end{equation}

 Putting $i = 1$ in (\ref{e:c7}) and
 using (\ref{e:b19}) and  (\ref{e:c5}), we  find
\begin{equation}\label{e:b20}
\omega_{1k}^{j}=\omega_{k1}^{j} = 0, \quad j \in C_s, \quad k
\in B\setminus C_s,  \ s\in \{ 1, 2, 3\}.
\end{equation}

 Combining (\ref{e:b19}) and (\ref{e:b17}), we obtain
\begin{equation}\begin{array}{lcl}\label{e:ba5}
\omega_{ji}^{1}=\omega_{ij}^{1}=0, \quad i\neq j,\quad i, j \in B.
\end{array}\end{equation}

 Now, using  (\ref{e:c12})$\sim$(\ref{e:ba5}) in (\ref{e:c4}),  we
complete the proof of the lemma.
\end{proof}

 We now evaluate $g(R(e_{1},e_{i})e_{1},e_{i}),
 \quad g(R(e_{1},e_{i})e_{i},e_{j})$ and $g(R(e_{i},e_{j})e_{i},e_{1})$ using Lemma \ref{l:2}, (\ref{e:b5}) and (\ref{e:c1}), and find the
following relations:
\begin{equation}\label{e:c14}
 e_{1}(\omega_{ii}^{1})- (\omega_{ii}^{1})^{2}= c+\lambda_{1} \lambda_{i}, \quad i \in
 B,
\end{equation}
\begin{equation}\label{e:c15}
 e_{1}(\omega_{ii}^{j})- \omega_{ii}^{j} \omega_{ii}^{1}= 0, \quad i\in C_s, \quad j \in B\setminus
 C_{s}, \quad s\in\{1, 2, 3\},
\end{equation}
and
 \begin{equation}\label{e:c16}
 e_{j}(\omega_{ii}^{1})+ \omega_{ii}^{j} \omega_{jj}^{1}-\omega_{ii}^{j} \omega_{ii}^{1}= 0, \quad i\in C_{s}, \quad j \in B\setminus
 C_{s}, \quad s\in\{1, 2, 3\},
\end{equation}
respectively.
Also, using (\ref{e:c3}), Lemma \ref{l:2}, and the fact that $[e_{i}\hspace{.1 cm}e_{1}](H)=0= \nabla_{e_{i}}e_{1}(H)-\nabla_{e_{1}}e_{i}(H),$  we find that
\begin{equation}\label{e:c17}
 e_{i}e_{1}(H)= 0, \ \ \mbox{for all}\ \  i\in B.
\end{equation}

\section{\textbf{Biconservative hypersurfaces in
 space forms with constant norm of second fundamental form}}

In this section we study biconservative hypersurfaces $M^n$ in
space forms $\overline M^{n+1}(c)$ with constant norm of second
fundamental form.
 We denote by $\beta$ the squared norm of the second
  fundamental form $h$. Then, using (\ref{e:c1}) we find that
\begin{equation}\label{e:g1}
\beta= \lambda_1^{2}+p\lambda_u^{2}+ q\lambda_v^{2}+r \lambda_w^2, \ u\in C_1, \ v\in C_2, \ w\in C_3.
\end{equation}

 We recall that $B=\{2,3,\dots,n\}$.  Then we  have the following:
\begin{lem}\label{l:3}
Let $M^n$ be a biconservative hypersurface with four distinct
principal curvatures in
  a space form $\overline M^{n+1}(c)$, having the shape operator given by
$(\ref{e:c1})$ with respect to a suitable orthonormal frame
$\{e_{1}, e_{2},\dots, e_{n}\}$.
If the second fundamental form is of constant norm, then
$$\omega^u_{vv}=0, \quad\mbox{for}\  u \in C_s,  \ \mbox{and}\  v\in B\setminus C_s, \ s\in \{1,2,3\}.$$
\end{lem}
\begin{proof} 
 We will prove the case $u\in C_1$ and $v\in C_2$  and by similar arguments one can prove all other cases. Let $w\in C_3.$ 
Differentiating (\ref{e:c11}) and
(\ref{e:g1}) with respect to $e_{u}$ and using (\ref{e:c3}), we get
\begin{equation}\label{e:d1}
 pe_{u}(\lambda_{u})+qe_{u}(\lambda_{v})+re_{u}(\lambda_{w})=0,
\end{equation}
and
\begin{equation}\label{e:d2}
 p\lambda_{u}e_{u}(\lambda_{u})+q\lambda_{v}e_{u}(\lambda_{v})+r\lambda_{w}e_{u}(\lambda_{w})=0,
\end{equation}
respectively.

Eliminating $e_{u}(\lambda_{u})$ from (\ref{e:d1}) and (\ref{e:d2}),
we find
\begin{equation}\label{e:d3}
q(\lambda_{v}-\lambda_{u})e_{u}(\lambda_{v})+r(\lambda_{w}-\lambda_{u})e_{u}(\lambda_{w})=0.
\end{equation}

Putting the value of $e_{u}(\lambda_{v})$ and $e_{u}(\lambda_{w})$
from (\ref{e:c6}) in (\ref{e:d3}), we obtain
\begin{equation}\label{e:d4}
 r(\lambda_{w}-\lambda_{u})^{2}\omega_{ww}^{u}+q(\lambda_{v}-\lambda_{u})^{2}\omega_{vv}^{u}=0.
\end{equation}

Differentiating (\ref{e:d4}) with respect to $e_{1}$ and using
(\ref{e:c6}) and (\ref{e:c15}), we have
\begin{equation}\label{e:d5}
\begin{array}{rcl}
 r[2(\lambda_{u}-\lambda_{1})\omega_{uu}^{1}+(2\lambda_{1}+\lambda_{u}-3\lambda_{w})\omega_{ww}^{1}](\lambda_w-\lambda_u)\omega_{ww}^{u}
 \\+q[2(\lambda_{u}-\lambda_{1})\omega_{uu}^{1}+(2\lambda_{1}+\lambda_{u}-3\lambda_{v})\omega_{vv}^{1}](\lambda_v-\lambda_u)\omega_{vv}^{u}=0.
\end{array}
\end{equation}

\smallskip
 We assume that 
$\omega_{vv}^{u}\neq 0$ and we will end up to contradiction.   Then the value of
the determinant formed by the coefficients of $\omega_{vv}^{u}$ and
$\omega_{ww}^{u}$ of the system (\ref{e:d4}) and (\ref{e:d5}) will be zero.
Therefore, we find
\begin{equation}\label{e:d6}
\begin{array}{rcl}
 2(\lambda_{1}-\lambda_{u})(\lambda_{v}-\lambda_{w})\omega_{uu}^{1}+(2\lambda_{1}+\lambda_{u}-3\lambda_{w})(\lambda_{u}-\lambda_{v})\omega_{ww}^{1}
 \\-(2\lambda_{1}+\lambda_{u}-3\lambda_{v})(\lambda_{u}-\lambda_{w})\omega_{vv}^{1}=0.
\end{array}
\end{equation}

We set $a_{1}=(\lambda_{w}-\lambda_{u})\omega_{vv}^{1}+(\lambda_{u}-\lambda_{v})\omega_{ww}^{1}+(\lambda_{v}-\lambda_{w})\omega_{uu}^{1}$.
Eliminating $\omega_{uu}^{1}$ from (\ref{e:d6}) and $a_1$, we get
\begin{equation}\label{e:d7}
\begin{array}{rcl}
 2(\lambda_{1}-\lambda_{u})a_1 =
 3(\lambda_{u}-\lambda_{v})(\omega_{vv}^{1}-\omega_{ww}^{1})
 (\lambda_{u}-\lambda_{w}).
\end{array}
\end{equation}

Now, we consider two cases.

\textbf{(i)} $a_1=0.$ Then, from (\ref{e:d7}), we obtain
\begin{equation}\label{e:d8}
\begin{array}{rcl}
 \omega_{ww}^{1}= \omega_{vv}^{1}.
\end{array}
\end{equation}

Differentiating (\ref{e:d8}) with respect to $e_1$ and using
(\ref{e:c14}) and (\ref{e:d8}), we find $\lambda_w=\lambda_v$, a
contradiction to four distinct principal curvatures.

\textbf{(ii)}  $a_1\neq 0.$ Then, differentiating (\ref{e:d7})
with respect to $e_u$ and using (\ref{e:d1}) and (\ref{e:c3}), we
get
\begin{comment}
\begin{equation}\label{e:d9}
\begin{array}{rcl}
 2p(\lambda_{1}-\lambda_{u})e_u(a_1) +2a_1(qe_u(\lambda_{v})+re_u(\lambda_w))=
 3\Big((-2qe_u(\lambda_{v})-re_u(\lambda_w))\\(\omega_{vv}^{1}-\omega_{ww}^{1})
 (\lambda_{u}-\lambda_{w})+p(\lambda_{u}-\lambda_{v})(e_u(\omega_{vv}^{1})-e_u(\omega_{ww}^{1}))
 (\lambda_{u}-\lambda_{w})\\+(\lambda_{u}-\lambda_{v})(\omega_{vv}^{1}-\omega_{ww}^{1})
 (-qe_u(\lambda_{v})-2re_u(\lambda_w))\Big).
\end{array}
\end{equation}
\end{comment}

\begin{equation}\label{e:d9}
\begin{array}{rcl}
 2p(\lambda_{1}-\lambda_{u})e_u(a_1) +2a_1(qe_u(\lambda_{v})+re_u(\lambda_w))=
 3\Big((-(q+p)e_u(\lambda_{v})-re_u(\lambda_w))\\(\omega_{vv}^{1}-\omega_{ww}^{1})
 (\lambda_{u}-\lambda_{w})+p(\lambda_{u}-\lambda_{v})(e_u(\omega_{vv}^{1})-e_u(\omega_{ww}^{1}))
 (\lambda_{u}-\lambda_{w})\\+(\lambda_{u}-\lambda_{v})(\omega_{vv}^{1}-\omega_{ww}^{1})
 (-qe_u(\lambda_{v})-(p+r)e_u(\lambda_w))\Big).
\end{array}
\end{equation}

Now, using (\ref{e:c6}) and (\ref{e:c16}) in (\ref{e:d9}), we obtain
\begin{equation}\label{e:d10}
\begin{array}{rcl}f_1\omega_{vv}^u
+f_2\omega_{ww}^u=2 p(\lambda_{1} - \lambda_{u}) e_u(a_1),
\end{array}
\end{equation}
where \begin{eqnarray} f_1&=&(\lambda_{u} - \lambda_{v})
\Big(3\omega_{ww}^{1}(q\lambda_v-(p+2q)\lambda_u+(p+q)\lambda_w)+3\omega_{vv}^{1}\big(2(p+q)\lambda_u-q\lambda_v\nonumber\\ &-&(2p+q)\lambda_w)+p(\lambda_u-\lambda_w)\big)+ 3p\omega_{uu}^{1}(\lambda_w-\lambda_u)+2a_1q\Big),\nonumber\\
f_2&=&(\lambda_{u} -
\lambda_{w})\Big(3\omega_{ww}^{1}\big((2p+r)\lambda_v+r\lambda_w-2(p+r)\lambda_u\big)+3\omega_{vv}^{1}((p+2r)\lambda_u-(p+r)\lambda_v\nonumber\\&-&r\lambda_w)\big)+3p\omega_{uu}^{1}(\lambda_u-\lambda_v)+2a_1r\Big).\nonumber \end{eqnarray}

Differentiating $a_1$ with respect to $e_u$ and using (\ref{e:c6}),
(\ref{e:c16}) and (\ref{e:d1}), we find
\begin{equation}\label{e:d11}
\begin{array}{rcl}
f_3\omega_{vv}^u+f_4\omega_{ww}^u+p(\lambda_{v}-\lambda_{w})
e_u(\omega_{uu}^{1}) = pe_u(a_1),
\end{array}
\end{equation}
where \begin{eqnarray}f_3&=&p \omega_{uu}^{1} (\lambda_{v} - \lambda_{w}) +
\omega_{vv}^{1} (q(\lambda_{v}- \lambda_{u}) +p(\lambda_{w}  -
\lambda_{u}))+(p+q)(\lambda_{u} - \lambda_{v}) \omega_{ww}^{1},\nonumber\\
f_4&=&p\omega_{uu}^{1} (\lambda_{v} - \lambda_{w})  +
(p+r)\omega_{vv}^{1} ( \lambda_{w} - \lambda_{u})+\omega_{ww}^{1}(p(
\lambda_{u} - \lambda_{v}) -r( \lambda_{w}-\lambda_u)).\nonumber\end{eqnarray}

Differentiating (\ref{e:c11}) with respect to $e_1$ and using
(\ref{e:c6}), we find
\begin{equation}\label{e:d12}
\begin{array}{rcl}
 p\omega_{uu}^{1} (\lambda_{u} - \lambda_{1}) +q \omega_{vv}^{1} (
\lambda_{v} -\lambda_{1})+r(\lambda_{w} - \lambda_{1})
\omega_{ww}^{1}=\frac{3ne_1(H)}{2}.
\end{array}
\end{equation}

Differentiating (\ref{e:d12}) with respect to $e_u$ and using
(\ref{e:c3}), (\ref{e:c6}), (\ref{e:c16}), (\ref{e:c17}) and
(\ref{e:d1}), we find
\begin{equation}\label{e:d13}
\begin{array}{rcl}
p(\lambda_{u}-\lambda_{1}) e_u(\omega_{uu}^{1}) =qf_5\omega_{vv}^u+
rf_6\omega_{ww}^u,
\end{array}
\end{equation}
where \begin{eqnarray}f_5=(\omega_{vv}^{1}-\omega_{uu}^{1})(\lambda_{1} +
\lambda_{u} - 2 \lambda_{v}),\quad
f_6=(\omega_{ww}^{1}-\omega_{uu}^{1})(\lambda_{1} + \lambda_{u} - 2
\lambda_{w}).\nonumber \end{eqnarray}

Eliminating $e_u(\omega_{uu}^{1})$ from (\ref{e:d11}) and
(\ref{e:d13}), we get
\begin{equation}\label{e:d14}
\begin{array}{rcl}
\Big((\lambda_{u}-\lambda_{1})f_3+(\lambda_{v} -
\lambda_{w})qf_5\Big)\omega_{vv}^u+\Big((\lambda_{u}-\lambda_{1})f_4+(\lambda_{v}
- \lambda_{w})rf_6\Big)\omega_{ww}^u\\=(\lambda_{u}-\lambda_{1})p
e_u(a_1).
\end{array}
\end{equation}

Eliminating $e_u(a_1)$ from (\ref{e:d10}) and (\ref{e:d14}), we find
\begin{equation}\label{e:d15}
\begin{array}{lcl}
f_7\omega_{vv}^u+f_8\omega_{ww}^u=0,
\end{array}
\end{equation}
where \begin{eqnarray}f_7=f_1+2(\lambda_{u}-\lambda_{1})f_3+2(\lambda_{v} -
\lambda_{w})qf_5,
\quad f_8=f_2+2(\lambda_{u}-\lambda_{1})f_4+2(\lambda_{v} -
\lambda_{w})rf_6.\nonumber\end{eqnarray}

Now, we simplify $f_7$ and $f_8$. Eliminating $\omega_{uu}^1$ from
$f_7$ using $a_1$, we get
%\begin{equation}
\begin{eqnarray}\label{e:d16}
(\lambda_w-\lambda_v)f_7&=& \Big(3 p \lambda _u^2-q \lambda _v^2+2(3p+2q)
\lambda _v \lambda _w-3 (p+q) \lambda _w^2-2\lambda _u ((3 p+ q) \lambda
_v\nonumber \\ &-&q \lambda _w)\Big)  (\omega _{vv}^1-\omega _{ww}^1)(\lambda
_u-\lambda _v)+a_1\Big(3 p \lambda _u^2+2 (p+q) \lambda _1( \lambda _v-\lambda_w) \\&-&p\lambda_u(5\lambda_v+\lambda_w)+\lambda_v(-2q\lambda_v+(3p+2q)\lambda_w)\Big).\nonumber
\end{eqnarray}
%\end{equation}

Eliminating $a_1$ from (\ref{e:d16}) using (\ref{e:d7}), we obtain
\begin{equation}\label{e:d17}
\begin{array}{lcl}
(\lambda _u-\lambda _v)^2 (\omega _{vv}^1-\omega _{ww}^1)g_1=-2 f_7
(\lambda _1-\lambda _u) (\lambda _v-\lambda _w),
\end{array}
\end{equation}
where \begin{eqnarray}g_1&=&3 p \lambda _u^2+4 \lambda _u( q\lambda_v-(3p+q)\lambda _w )+2\lambda_1(3p\lambda_u +q \lambda
_v-(3 p+ q) \lambda _w)\nonumber\\&+&3\lambda _w ((3 p+2 q) \lambda _w-2 q \lambda _v)\nonumber\end{eqnarray}

Similarly, eliminating $\omega_{uu}^1$ from $f_8$ using $a_1$, we
get
\begin{eqnarray}\label{e:d18}
a_1\Big( 2 \lambda _1 (p+r) \left(\lambda _v-\lambda _w\right)+\lambda _w \left(2 r \lambda _w-(3 p+2 r) \lambda _v\right)-3 p \lambda _u^2+p \lambda _u \left(\lambda _v+5 \lambda _w\right)\Big)\nonumber\\+f_8 (\lambda _v-\lambda _w)=(\lambda
_u-\lambda _w)  (\omega _{vv}^1-\omega _{ww}^1)\Big(2 \lambda _u (r \lambda _v-(3 p+r) \lambda _w)\\-3 (p+r) \lambda _v^2+2 (3 p+2 r) \lambda _v \lambda _w+3 p \lambda _u^2-r \lambda _w^2\Big).\nonumber
\end{eqnarray}

Eliminating $a_1$ from (\ref{e:d18}) using (\ref{e:d7}), we obtain
\begin{equation}\label{e:d19}
\begin{array}{lcl}
2 f_8 (\lambda _1-\lambda _u) (\lambda _v-\lambda _w)=(\lambda
_u-\lambda _w)^2 (\omega _{vv}^1-\omega _{ww}^1)g_2,
\end{array}
\end{equation}
where \begin{eqnarray}g_2&=&-4 \lambda _u ((3 p+r) \lambda _v-r \lambda _w)+2 \lambda _1 (-(3 p+r) \lambda _v+3 p \lambda _u+r \lambda _w)+3 \lambda _v ((3 p+2 r) \lambda _v\nonumber\\&-&2 r \lambda _w)+3 p \lambda _u^2.\nonumber \end{eqnarray}

Eliminating $f_7$ and $f_8$ from (\ref{e:d15}) using (\ref{e:d17})
and (\ref{e:d19}), we obtain
\begin{equation}\label{e:d20}
\begin{array}{lcl}
g_1 \left(\lambda _u-\lambda _v\right)^2 \omega _{vv}^u-g_2
\left(\lambda _u-\lambda _w\right)^2 \omega _{ww}^u=0.
\end{array}
\end{equation}

The value of
the determinant formed by the coefficients of $\omega_{vv}^{u}$ and
$\omega_{ww}^{u}$ of the system (\ref{e:d4}) and (\ref{e:d20}) will be zero. Therefore,  we have
\begin{equation}\label{e:d21}
\begin{array}{lcl}
q g_2 +r g_1=0,
\end{array}
\end{equation}
which on substitution for $g_1, g_2$, gives
\begin{eqnarray}\label{e:n1}
p (q+r) \lambda _u^2+2 \lambda _1 p (q+r) \lambda _u+\lambda _w \left(-2 \lambda _1 p r-4 p r \lambda _u-4 q r \lambda _v\right)\\+\lambda _w^2 (3 p r+2 q r)-4 p q \lambda _u \lambda _v+3 p q \lambda _v^2-2 \lambda _1 p q \lambda _v+2 q r \lambda _v^2=0.\nonumber
\end{eqnarray}

 Now,  eliminating $\lambda_w$ from (\ref{e:g1}) and (\ref{e:n1})
using (\ref{e:c11}), we find
\begin{equation}\label{e:d22}
\begin{array}{lcl}
-r \beta +(9+r) \lambda _1^2+6 p \lambda _1 \lambda _u+(p^2 +p r)
\lambda _u^2+2q(3 \lambda _1+ p \lambda _u) \lambda _v+(q^2+q r)
\lambda _v^2=0,
\end{array}
\end{equation}
and
\begin{equation}\label{e:n2}
\begin{array}{lcl}
b_0+b_1 \lambda _v+b_2 \lambda _v^2=0,
\end{array}
\end{equation}
where \begin{eqnarray}b_0&=&p \lambda _u^2 \left(3 p^2+2 p (q+2 r)+r (q+r)\right)+3 \lambda _1^2 (p (2 r+9)+6 q)\nonumber\\&+&2 \lambda _1 p \lambda _u (p (r+9)+(r+6) (q+r)),\nonumber\\
b_1&=&6 \lambda _1 q (3 p+2 (q+r))+2 p q \lambda _u (3 p+2 (q+r)),\nonumber\\
b_2&=&q (q+r) (3 p+2 (q+r)).\nonumber \end{eqnarray}

  Eliminating $\lambda_v$ from (\ref{e:n2}) 
using (\ref{e:d22}), we obtain
\begin{eqnarray}\label{e:z1}
3 \beta  p+\lambda _1^2 (3 p-2 (q+r+9))+2 \lambda _1 p \lambda _u (p+q+r)-p \lambda _u^2 (p+q+r)\\+2 \beta  (q+r)=0.\nonumber\end{eqnarray}
\begin{comment}
  Equating to zero the resultant of the polynomial equations
(\ref{e:d22}) and (\ref{e:n2}) with respect $\lambda _v$, we find the
following polynomial equation in terms of $\lambda _1, \lambda _u$:
\begin{eqnarray}\label{e:n3}
%\begin{array}{lcl}
F(\lambda _1, \lambda _u)=\sum_{k=0}^{8}c_k \lambda _u^k=0,
\end{eqnarray}
%\end{equation}
where the coefficients $c_k$'s ($k=0,1,\dots,8$) of $\lambda_u^k$ are functions of $\lambda_1$ and \begin{eqnarray}F&=&\beta^4 q^4 r^4 (q+r)^2 (3 p+2 q+2r)^2-81 p^2 q^6 r^3 \beta ^3 \lambda _1^2-108 p q^7 r^3 \beta ^3\lambda _1^2
\nonumber\\&-&36 p^2 q^7 r^3 \beta ^3\lambda _1^2 -36 q^8 r^3 \beta
^3 \lambda _1^2+12 p q^8 r^3 \beta ^3 \lambda _1^2-4 p^2 q^8 r^3
\beta ^3 \lambda _1^2+24 q^9 r^3 \beta ^3 \lambda _1^2
+\langle\langle 2624\rangle\rangle \nonumber\\&+&[p^2 q^3 r^3 (q + r)^2 (p +
q + r)^3 (9 p^3 - 6 p^2 (q + r) +
   25 q r (q + r) + p (q^2 - 37 q r + r^2))
]\lambda _u^8.\nonumber\end{eqnarray}
\end{comment}

 Differentiating (\ref{e:z1}) with respect to $e_u$ and using (\ref{e:c3}), we get
\begin{equation}\label{e:n4}
2p(p+q+r)( \lambda _1- \lambda _u )e_u(\lambda_u)=0, \end{equation}
whereby, we get $e_u(\lambda_u)=0$.
Therefore, from (\ref{e:d1}) and (\ref{e:c6}), we obtain
\begin{equation}\label{e:d23}
\begin{array}{lcl}
q(\lambda_{v} - \lambda_{u})\omega_{vv}^u+r(\lambda_{w}
-\lambda_{u})\omega_{ww}^u=0.
\end{array}
\end{equation}

The value of
the determinant formed by the coefficients of $\omega_{vv}^{u}$ and
$\omega_{ww}^{u}$ of the system (\ref{e:d4}) and (\ref{e:d23}) will be zero. Therefore,  we get
\begin{equation}\label{e:d24}
\begin{array}{lcl}
(\lambda_{v} - \lambda_{u})(\lambda_{w} - \lambda_{u})(\lambda_{v} -
\lambda_{w})=0,
\end{array}
\end{equation}
which gives a contradiction to four distinct principal curvatures.
Hence, we obtain that $\omega_{vv}^u=0$ and the 
 proof of the lemma
%\ref{l:3} 
is completed.
\end{proof}

 Next, we have:
\begin{lem}{\rm \label{l:4}}
Under the assumtions of Lemma \ref{l:3},
 let 
 $a_{1}=(\lambda_{w}-\lambda_{u})\omega_{vv}^{1}+(\lambda_{u}-\lambda_{v})\omega_{ww}^{1}+(\lambda_{v}-\lambda_{w})\omega_{uu}^{1}.$  Then,
 
 \smallskip
\noindent
{\rm \textbf{(a)}} If $a_1\neq 0$, it is
\begin{equation}\label{e:d43}\omega_{wv}^{u}=\omega_{vw}^{u}=
\omega_{wu}^{v}=\omega_{uw}^{v}=\omega_{vu}^{w}=\omega_{uv}^{w}=0,
\quad \mbox{where}\ u\in C_1, v\in C_2, w\in C_3.
\end{equation} 
{\rm \textbf{(b)}} If $a_1=0$, we have that
\begin{equation}\label{e:d44} \omega_{ii}^{1}=\alpha\lambda_i+\phi,\quad e_1(\alpha)=\alpha\phi+\lambda_{1}(1+\alpha^2),\quad
e_1(\phi)=\phi^2+\alpha\lambda_1\phi+c,\end{equation} for some
smooth functions $\alpha$ and $\phi$,  and for $i\in B$.
\end{lem}
\begin{proof}  \textbf{(a)} Let $a_1\neq 0.$  Evaluating
$g(R(e_{v},e_{u})e_{w},e_{1})$ using (\ref{e:b5}), (\ref{e:c1}) and
Lemma \ref{l:2} and Lemma \ref{l:3}, we find
\begin{equation}\label{e:d45}
\omega_{vu}^{w}(\omega_{uu}^{1}-\omega_{ww}^{1})=\omega_{uv}^{w}(\omega_{vv}^{1}-\omega_{ww}^{1}).
\end{equation}

Putting $j=w, k=u, i=v$ in (\ref{e:c7}), we get
 \begin{equation}\label{e:d46}
(\lambda_{u}-\lambda_{w})\omega_{vu}^{w}=(\lambda_{v}-\lambda_{w})\omega_{uv}^{w}.
\end{equation}

The value of the determinant formed by the coefficients of
$\omega_{vu}^{w}$ and $\omega_{uv}^{w}$ in (\ref{e:d45}) and
(\ref{e:d46}) is $a_1\neq 0$, hence
$\omega_{vu}^{w}=0=\omega_{uv}^{w}$. Also, from (\ref{e:c5}), we get
$\omega_{vw}^{u}=-\omega_{vu}^{w}$ and
$\omega_{uv}^{w}=-\omega_{uw}^{v}$. Consequently, we obtain
$\omega_{vw}^{u}=0$, and $\omega_{uw}^{v}=0$, which together with
(\ref{e:c7}) gives $\omega_{wv}^{u}=0$, and $\omega_{wu}^{v}=0$.\\
\\
\textbf{(b)} Let $a_1=0.$ Then, we have
\begin{equation}\label{e:d47}
\frac{\omega_{uu}^{1}-\omega_{vv}^{1}}{\lambda_{u}-\lambda_{v}}=\frac{\omega_{ww}^{1}-\omega_{vv}^{1}}{\lambda_{w}-\lambda_{v}}
=\frac{\omega_{uu}^{1}-\omega_{ww}^{1}}{\lambda_{u}-\lambda_{w}}=\alpha,
\end{equation}
for some smooth function $\alpha$.

From (\ref{e:d47}), we get
\begin{equation}\label{e:d48}
\omega_{ii}^{1}=\alpha\lambda_{i}+\phi,
\end{equation}
for some smooth function $\phi$.

Differentiating (\ref{e:d48}) with respect to $e_1$ and using
(\ref{e:c6}), (\ref{e:c14}) and (\ref{e:d48}), we find
\begin{equation}\label{e:d49}
e_1(\alpha)=\alpha\phi+\lambda_{1}(1+\alpha^2),\quad
e_1(\phi)=\phi^2+\alpha\lambda_1\phi+c,
\end{equation}
whereby completing the proof of the lemma.\end{proof}

\section{ \textbf{Proof of
Theorem 1.1}}

Depending upon principal curvatures, we consider the following
cases.

\medskip
\noindent
 \textbf{ Case 1.} \emph{The case of four distinct principal
curvatures }

\smallskip
From (\ref{e:c7}) and
(\ref{e:c5}), we obtain
\begin{equation}\label{e:d55}
(\lambda_{u}-\lambda_{v})\omega_{wu}^{v}=(\lambda_{w}-\lambda_{v})\omega_{uw}^{v}=(\lambda_{u}-\lambda_{w})\omega_{vu}^{w}.
\end{equation}

From (\ref{e:d55}) and (\ref{e:c5}), we find
\begin{equation}\label{e:d56}
\omega_{vw}^{u}\omega_{wv}^{u}+\omega_{wu}^{v}\omega_{uw}^{v}+\omega_{vu}^{w}\omega_{uv}^{w}=0.
\end{equation}

Evaluating $g(R(e_{u},e_{v})e_{u},e_{v})$,
$g(R(e_{u},e_{w})e_{u},e_{w})$ and $g(R(e_{v},e_{w})e_{v},e_{w})$,
using (\ref{e:b5}), (\ref{e:c1}), (\ref{e:d56}) and Lemmas \ref{l:2} and 
\ref{l:3}, we find that
\begin{equation} \label{e:d50}
-\omega_{uu}^{1}\omega_{vv}^{1}
+\sum\nolimits_{k\in B\setminus\{C_1,C_2\}}2\omega_{uv}^{k}\omega_{vu}^{k}=c+
\lambda_{u} \lambda_{v},
\end{equation}
\begin{equation} \label{e:d51}
-\omega_{uu}^{1}\omega_{ww}^{1}
+\sum\nolimits_{k\in B\setminus\{C_1,C_3\}}2\omega_{uw}^{k}\omega_{wu}^{k}=c+
\lambda_{u} \lambda_{w},
\end{equation}
\begin{equation} \label{e:d52}
-\omega_{vv}^{1}\omega_{ww}^{1}
+\sum\nolimits_{k\in B\setminus\{C_2,C_3\}}2\omega_{vw}^{k}\omega_{wv}^{k}=c+
\lambda_{v} \lambda_{w},
\end{equation}
respectively.

Simplifying (\ref{e:d50}), (\ref{e:d51}) and (\ref{e:d52}), we obtain
 \begin{equation} \label{e:d50a}
-\omega_{uu}^{1}\omega_{vv}^{1}
+2 r\omega_{uv}^{w}\omega_{vu}^{w}=c+
\lambda_{u} \lambda_{v},
\end{equation}
\begin{equation} \label{e:d51a}
-\omega_{uu}^{1}\omega_{ww}^{1}
+2 q\omega_{uw}^{v}\omega_{wu}^{v}=c+
\lambda_{u} \lambda_{w},
\end{equation}
\begin{equation} \label{e:d52a}
-\omega_{vv}^{1}\omega_{ww}^{1}
+2 p\omega_{vw}^{u}\omega_{wv}^{u}=c+
\lambda_{v} \lambda_{w},
\end{equation}
respectively.

 Depending upon $a_1$, we consider the following cases.

\smallskip
\textbf{Subcase A.} Assume that $a_1\neq 0$. Using (\ref{e:d43}) in
(\ref{e:d50}), (\ref{e:d51}) and (\ref{e:d52}), we obtain
\begin{equation} \label{e:d53}
-\omega_{uu}^{1}\omega_{vv}^{1}=c+\lambda_{u} \lambda_{v},\quad
-\omega_{uu}^{1}\omega_{ww}^{1}=c+\lambda_{u}
\lambda_{w},\quad-\omega_{vv}^{1}\omega_{ww}^{1}=c+\lambda_{v}
\lambda_{w}.
\end{equation}

From (\ref{e:d53}), we get
\begin{equation}\label{e:d54}
-(c+\lambda_{w} \lambda_{v})(\omega_{uu}^{1})^{2}= (c+\lambda_{u}
\lambda_{v})(c+\lambda_{u} \lambda_{w}).
\end{equation}

Differentiating (\ref{e:c11}) and (\ref{e:g1}) with respect to $e_1$
and using (\ref{e:c6}), we find
\begin{equation}\label{e:n11}
p(\lambda_{u}- \lambda_{1})\omega_{uu}^{1}+q(\lambda_{v}-
\lambda_{1})\omega_{vv}^{1}+r(\lambda_{w}-
\lambda_{1})\omega_{ww}^{1}= -3e_1(\lambda_{1}),
\end{equation}
and
\begin{equation}\label{e:n12}
p\lambda_{u}(\lambda_{u}-
\lambda_{1})\omega_{uu}^{1}+q\lambda_{v}(\lambda_{v}-
\lambda_{1})\omega_{vv}^{1}+r\lambda_{w}(\lambda_{w}-
\lambda_{1})\omega_{ww}^{1}= -\lambda_{1}e_1(\lambda_{1}),
\end{equation}
respectively.

Eliminating $e_1(\lambda_1)$ from (\ref{e:n12}) using (\ref{e:n11}),
we obtain
\begin{equation}\label{e:n13}
p(3\lambda_{u}- \lambda_{1})(\lambda_{u}-
\lambda_{1})\omega_{uu}^{1}+q(3\lambda_{v}-
\lambda_{1})(\lambda_{v}-
\lambda_{1})\omega_{vv}^{1}+r(3\lambda_{w}-
\lambda_{1})(\lambda_{w}- \lambda_{1})\omega_{ww}^{1}=0.
\end{equation}

Multiplying (\ref{e:n13}) with $\omega_{uu}^{1}$ and using
(\ref{e:d53}), we find
\begin{eqnarray}\label{e:n14}%\begin{array}{rcl}
p(3\lambda_{u}- \lambda_{1})(\lambda_{u}-
\lambda_{1})(\omega_{uu}^{1})^2&=&q(3\lambda_{v}-
\lambda_{1})(\lambda_{v}-
\lambda_{1})(c+\lambda_u\lambda_v)\nonumber \\&+&r(3\lambda_{w}-
\lambda_{1})(\lambda_{w}- \lambda_{1})(c+\lambda_u\lambda_w).
%\end{array}
\end{eqnarray}

Eliminating $(\omega_{uu}^{1})^2$ from (\ref{e:n14}) using
(\ref{e:d54}), we obtain
\begin{equation}\label{e:n15}\begin{array}{rcl}
-p(3\lambda_{u}- \lambda_{1})(\lambda_{u}-
\lambda_{1})(c+\lambda_u\lambda_w)(c+\lambda_u\lambda_v)=\big(
q(3\lambda_{v}- \lambda_{1})(\lambda_{v}-
\lambda_{1})\\(c+\lambda_u\lambda_v)+r(3\lambda_{w}-
\lambda_{1})(\lambda_{w}-
\lambda_{1})(c+\lambda_u\lambda_w)\big)(c+\lambda_v\lambda_w).
\end{array}\end{equation}

Eliminating $\lambda_w$ from (\ref{e:n15}) using (\ref{e:c11}), we
obtain
\begin{equation}\label{e:n16}\begin{array}{rcl}
v_0+v_1\lambda_v+v_2\lambda_v^2+v_3\lambda_v^3+v_4\lambda_v^4+v_5\lambda_v^5=0,\end{array}\end{equation}
where
\begin{eqnarray}v_0&=&c r (\lambda _1^2 (c r (r (p+q+12)+r^2+27)-p ((p+12) r+r^2+81) \lambda _u^2)+9 \lambda _1 (2 c p r \lambda _u\nonumber\\&-&p (3 p+r) \lambda _u^3)+3 p (p+r) \lambda _u^2 (c r-p \lambda _u^2)-3 \lambda _1^3 ((p+12) r+r^2+27) \lambda _u), \nonumber \\ v_1&=&\lambda _1 (18 c^2 q r^2-c p r \lambda _u^2 (p (4 r+27)+4 q r+54 q+4 r^2)+p (4 p^2 (r+9)+4 p r^2-9 r^2) \lambda _u^4)\nonumber\\&+&3 p \lambda _u (2 c^2 q r^2-c r \lambda _u^2 (p^2+3 p q+r (q-r))+(p^3-p r^2) \lambda _u^4-3 \lambda _1^3 (c r ((q+12) r\nonumber\\&+&r^2+27)-p (r^2+36 r+108) \lambda _u^2)-\lambda _1^2 (c r \lambda _u (p (2 (q+12) r+81)+3 q (8 r+27))\nonumber\\&-&6 p (3 p (2 r+9)+2 r^2) \lambda _u^3)+9 \lambda _1^4 (r^2+12 r+27) \lambda _u,\nonumber \\
v_2&=&q (3 (c^2 r^2 (q+r)-3 c p r (p+q) \lambda _u^2+(4 p^3-p r^2) \lambda _u^4)-\lambda _1^2 (c r ((q+12) r+r^2+81)\nonumber\\&-&36 p (2 r+9) \lambda _u^2)-\lambda _1 (c r \lambda _u (p (4 r+54)+4 q r+27 q+4 r^2)-4 p (3 p (r+9)+r^2) \lambda _u^3)\nonumber\\&+&3 \lambda _1^3 (r^2+36 r+108) \lambda _u),\nonumber \\
v_3&=&q (-3 c r \lambda _u (3 p q+p r+q^2-r^2)+\lambda _1 (4 p (3 q (r+9)+r^2) \lambda _u^2-9 c r (3 q+r))+18 p^2 q \lambda _u^3\nonumber\\&+&6 \lambda _1^2 (3 q (2 r+9)+2 r^2) \lambda _u),\nonumber \\
v_4&=&q (\lambda _1 (4 q^2 (r+9)+4 q r^2-9 r^2) \lambda _u-3 (c q r (q+r)+p (r^2-4 q^2) \lambda _u^2)),\nonumber \\ v_5&=&3 q^2(q^2 - r^2 )\lambda _u.\nonumber \end{eqnarray}

Eliminating $\lambda_w$ from (\ref{e:g1}) using (\ref{e:c11}), we
obtain
\begin{equation}\label{e:n17}\begin{array}{rcl}
v_6+v_7\lambda_v+v_8\lambda_v^2=0,\end{array}\end{equation} where
\begin{eqnarray}v_6=-r \beta +9 \lambda _1^2+r \lambda _1^2+6 p \lambda _1 \lambda
_u+p^2 \lambda _u^2+p r \lambda _u^2,\quad v_7=6 q \lambda _1+2 p q
\lambda _u, \quad v_8=q^2+q r.\nonumber\end{eqnarray}

Equations (\ref{e:n16}) and (\ref{e:n17}) have a common root $\lambda _v$, so
their resultant with respect $\lambda_v$ vanish. Hence,
\begin{eqnarray}\label{e:n18}
v_5^2 v_6^5-v_4 v_5 v_6^4 v_7+v_3 v_5 v_6^3 v_7^2-v_2 v_5 v_6^2
v_7^3+v_1 v_5 v_6 v_7^4-v_0 v_5 v_7^5+v_4^2 v_6^4 v_8\nonumber\\-2 v_3 v_5
v_6^4 v_8-v_3 v_4 v_6^3 v_7 v_8+3 v_2 v_5 v_6^3 v_7 v_8+v_2 v_4
v_6^2 v_7^2 v_8-4 v_1 v_5 v_6^2 v_7^2 v_8\nonumber\\-v_1 v_4 v_6 v_7^3 v_8+5
v_0 v_5 v_6 v_7^3 v_8+v_0 v_4 v_7^4 v_8+v_3^2 v_6^3 v_8^2-2 v_2 v_4
v_6^3 v_8^2+2 v_1 v_5 v_6^3 v_8^2\nonumber\\-v_2 v_3 v_6^2 v_7 v_8^2+3 v_1
v_4 v_6^2 v_7 v_8^2-5 v_0 v_5 v_6^2 v_7 v_8^2+v_1 v_3 v_6 v_7^2
v_8^2-4 v_0 v_4 v_6 v_7^2 v_8^2\\-v_0 v_3 v_7^3 v_8^2+v_2^2 v_6^2
v_8^3-2 v_1 v_3 v_6^2 v_8^3+2 v_0 v_4 v_6^2 v_8^3-v_1 v_2 v_6 v_7
v_8^3+3 v_0 v_3 v_6 v_7 v_8^3\nonumber\\+v_0 v_2 v_7^2 v_8^3+v_1^2 v_6
v_8^4-2 v_0 v_2 v_6 v_8^4-v_0 v_1 v_7 v_8^4+v_0^2 v_8^5=0,\nonumber
%\end{array}
\end{eqnarray}
which is a polynomial equation 
\begin{eqnarray}
\label{e:n19} G(\lambda_1, \lambda_u)=0,
\end{eqnarray} 
for
 $\lambda_1, \lambda_u$.
 
 \begin{comment}
 of degree $12$, where
\begin{eqnarray}\label{e:n19}
%\begin{array}{rcl}
G(\lambda_1, \lambda_u)=q^3 r^4 (q+r) \Big(9 c^4 q^6 r^2 \beta ^2+36
c^4 q^5 r^3 \beta ^2+54 c^4 q^4 r^4 \beta ^2+36 c^4 q^3 r^5 \beta
^2\nonumber\\+9 c^4 q^2 r^6 \beta ^2-18 c^3 q^5 r^2 \beta ^3-54 c^3 q^4 r^3
\beta ^3-54 c^3 q^3 r^4 \beta ^3-18 c^3 q^2 r^5 \beta ^3\nonumber\\+9 c^2 q^4
r^2 \beta ^4+\langle\langle 3810\rangle\rangle +171 p^6 q^2 r^2
\lambda _u^{12}+72 p^5 q^3 r^2 \lambda _u^{12}+9 p^4 q^4 r^2 \lambda
_u^{12}\\+36 p^7 r^3 \lambda _u^{12}+90 p^6 q r^3 \lambda _u^{12}+72
p^5 q^2 r^3 \lambda _u^{12}+18 p^4 q^3 r^3 \lambda _u^{12}+9 p^6 r^4
\lambda _u^{12}\nonumber\\+18 p^5 q r^4 \lambda _u^{12}+9 p^4 q^2 r^4 \lambda
_u^{12}\Big)=0.\nonumber
%\end{array}
\end{eqnarray}
\end{comment}
Differentiating (\ref{e:n19}) with respect to $e_1$, we get
\begin{equation}\label{e:n20}\begin{array}{rcl}
G_1e_1(\lambda_1)+G_ue_1(\lambda_u)=0,\end{array}\end{equation}
where $G_1=\frac{\partial G}{\partial \lambda_1}, G_u=\frac{\partial
G}{\partial \lambda_u}$.

Eliminating $e_1(\lambda_1)$ from (\ref{e:n20}) using (\ref{e:n11})
and using (\ref{e:c6}), we find
\begin{equation}\label{e:n21}\begin{array}{rcl}
(3G_u-pG_1)(\lambda_u-\lambda_1)\omega_{uu}^1=\Big(q(\lambda_v-\lambda_1)\omega_{vv}^1+r(\lambda_w-\lambda_1)\omega_{ww}^1\Big)G_1.\end{array}\end{equation}

Multiplying (\ref{e:n21}) with $\omega_{uu}^{1}$ and using
(\ref{e:d53}) and (\ref{e:d54}), we obtain
\begin{equation}\label{e:n22}\begin{array}{rcl}
L(c+\lambda_u\lambda_v)(c+\lambda_u\lambda_w)
=\Big(q(\lambda_v-\lambda_1)(c+\lambda_u\lambda_v)+r(\lambda_w-\lambda_1)\\(c+\lambda_u\lambda_w)\Big)(c+\lambda_w\lambda_v),\end{array}\end{equation}
where $L=\frac{(3G_u-pG_1)(\lambda_u-\lambda_1)}{G_1}$.

Eliminating $\lambda_w$ from (\ref{e:n22}) using (\ref{e:c11}), we
get
\begin{equation}\label{e:n23}\begin{array}{rcl}v_9+v_{10}\lambda_v+v_{11}\lambda_v^2+v_{12}\lambda_v^3+v_{13}\lambda_v^4=0,\end{array}\end{equation}
where
%\begin{center}
%\begin{array}{rcl}
\begin{eqnarray}
v_9&=&c^2 L r^2+(3 c^2 r^2 +c^2 q r^2 +c^2 r^3) \lambda _1+c^2 p r^2 \lambda _u-3 c L r \lambda _1 \lambda _u+(-9 c r -3 c
r^2) \lambda _1^2 \lambda _u\nonumber\\&-&c L p r \lambda _u^2 +(-6 c p r-c p r^2 )\lambda _1 \lambda _u^2-c p^2 r \lambda
_u^3,\nonumber\\
v_{10}&=&(-3
-q-r)3cr\lambda _1^2+(r-q)cLr \lambda _u-(6 p +6  q + p q +p r)cr \lambda _1
\lambda _u+(27+9 r )\lambda _1^3 \lambda _u\nonumber\\&-&(
p+2 q) cpr\lambda _u^2-3 L r \lambda _1 \lambda
_u^2 +(27 p +6 p r )\lambda _1^2 \lambda
_u^2-L p r \lambda _u^3+(9 + r )p^2\lambda
_1 \lambda _u^3+p^3 \lambda _u^4,\nonumber\\
v_{11}&=&[(-3-cq-r )\lambda _1 -(p+q+r)\lambda _u]cqr+[(9+r )3\lambda _1+(18+r )p \lambda _u]q\lambda _1 \lambda _u\nonumber\\&+&(3 p^2
\lambda _u-Lr)q \lambda _u^2,\nonumber\\ v_{12}&=&9
q^2 \lambda _1 \lambda _u +3 q r \lambda _1 \lambda _u+3 p q^2
\lambda _u^2+p q r \lambda _u^2,\quad\quad
v_{13}=q^3 \lambda _u
+q^2 r \lambda _u.\nonumber\end{eqnarray}

Equations (\ref{e:n23}) and (\ref{e:n17})
have a common root $\lambda _v$, so
their resultant with respect $\lambda_v$ vanish. Therefore, we find
\begin{equation}\label{e:n24}\begin{array}{lcl}
\mathcal{G}(\lambda_1, \lambda_u)=0,
\end{array}\end{equation}
which is a polynomial equation in terms of $\lambda_1, \lambda_u$.
\begin{comment}
where \begin{eqnarray}\mathcal{G}(\lambda_1, \lambda_u)&=&c^4 L^2 q^8 r^4+4 c^4 L^2
q^7 r^5+6 c^4 L^2 q^6 r^6+4 c^4 L^2 q^5 r^7+c^4 L^2 q^4
r^8+\langle\langle 2906\rangle\rangle \nonumber\\&+&p^6 q^3 r^5 \lambda _u^{10}+3
p^5 q^4 r^5 \lambda _u^{10}+2 p^4 q^5 r^5 \lambda _u^{10}+p^5 q^3
r^6 \lambda _u^{10}+p^4 q^4 r^6 \lambda _u^{10}.\nonumber\end{eqnarray}
\end{comment}

Rewrite (\ref{e:n19}) and (\ref{e:n24}) as polynomials
$G_{\lambda_1}(\lambda_u), \mathcal{G}_{\lambda_1}(\lambda_u)$ of
$\lambda_u$ with coefficients in the polynomial ring $\mathbb{R}[\lambda_1]$ over
real field. According to Lemma 2.1,  equations
$G_{\lambda_1}(\lambda_u)=0$ and
$\mathcal{G}_{\lambda_1}(\lambda_u)=0$ have a common root if and
only if
\begin{equation}\label{e:n25}\begin{array}{lcl}
\mathfrak{R}(G_{\lambda_1}(\lambda_u),\mathcal{G}_{\lambda_1}(\lambda_u))=0,
\end{array}\end{equation}
which is a polynomial of $\lambda_1$ with real coefficients.  Then
(\ref{e:n25}) shows that $\lambda_1$ must be a constant, a
contradiction.

\smallskip
\textbf{Subcase B.} Assume that $a_1= 0$.
\begin{comment}From (\ref{e:c7}) and
(\ref{e:c5}), we obtain
\begin{equation}\label{e:d55}
(\lambda_{u}-\lambda_{v})\omega_{wu}^{v}=(\lambda_{w}-\lambda_{v})\omega_{uw}^{v}=(\lambda_{u}-\lambda_{w})\omega_{vu}^{w}.
\end{equation}
From (\ref{e:d55}) and (\ref{e:c5}), we find
\begin{equation}\label{e:d56}
\omega_{vw}^{u}\omega_{wv}^{u}+\omega_{wu}^{v}\omega_{uw}^{v}+\omega_{vu}^{w}\omega_{uv}^{w}=0.
\end{equation}
\end{comment}
Adding  (\ref{e:d50a}), (\ref{e:d51a}) and (\ref{e:d52a}) and using
(\ref{e:d56}), we obtain
\begin{equation} \label{e:d57}
pq(\omega_{uu}^{1}\omega_{vv}^{1}+c+\lambda_{u}
\lambda_{v})+pr(\omega_{uu}^{1}\omega_{ww}^{1}+c+\lambda_{u}
\lambda_{w})+qr(\omega_{vv}^{1}\omega_{ww}^{1}+c+\lambda_{v}
\lambda_{w})=0. \end{equation}

Using  (\ref{e:d44}) in (\ref{e:d57}) in Lemma 4.2 we find
\begin{eqnarray}\label{e:d58a}
%3\phi^2+2\phi\alpha(\lambda_u+\lambda_v+\lambda_w)+(\alpha^2+1)(\lambda_u\lambda_v+\lambda_u\lambda_w+\lambda_v\lambda_w)+3c=0.
 \left(c+\phi ^2\right) (p q+p r+q r)+\left(\alpha ^2+1\right) 
 \left(p q \lambda _u \lambda _v+p r \lambda _u \lambda _w+q r \lambda _v \lambda _w\right)\\ 
 +\alpha  \phi  \left(p q \lambda _u+p q \lambda _v+p r \lambda _u+p r \lambda _w+q r 
 \lambda _v+q r \lambda _w\right)=0. \nonumber\end{eqnarray}
Using  (\ref{e:d44}) in (\ref{e:n13}), we get
\begin{equation}\label{e:d58}\begin{array}{rcl}
p(3\lambda _u-\lambda _1)(\lambda _u-\lambda _1)(\alpha \lambda
_u+\phi )+q\left(3\lambda _v-\lambda _1\right)\left(\lambda
_v-\lambda _1\right)\left(\alpha \lambda _v+\phi
\right)\\+r\left(3\lambda _w-\lambda _1\right)\left(\lambda
_w-\lambda _1\right)\left(\alpha \lambda _w+\phi \right)=0.
\end{array}\end{equation}

 On the other hand, using (\ref{e:c11}),
(\ref{e:g1}) and (\ref{e:d44}) in (\ref{e:d12}), we obtain
\begin{equation}\label{e:d59}
(n+2)\lambda_1\phi-\alpha(\beta+2\lambda_1^2)=3e_1(\lambda_1).
\end{equation}

Eliminating $\lambda_w$ from (\ref{e:g1}), {(\ref{e:d58a})} and
(\ref{e:d58}) using (\ref{e:c11}), we find
\begin{equation}\label{e:n26}\begin{array}{rcl}
9 \lambda _1^2+r \lambda _1^2+6 p \lambda _1 \lambda _u+p^2 \lambda
_u^2+p r \lambda _u^2+6 q \lambda _1 \lambda _v+2 p q \lambda _u
\lambda _v\\+q^2 \lambda _v^2+q r \lambda _v^2-r
\beta=0,\end{array}\end{equation}
\begin{equation}\label{e:n27}\begin{array}{rcl}
{\left(c+\phi ^2\right) (p q+p r+q r)-3 \alpha  \lambda _1 \phi  (p+q)-(\alpha  \phi  (p-r)+3 (\alpha ^2+1) \lambda _1) p \lambda _u}\\ 
{-(\alpha ^2+1) p^2 \lambda _u^2+\lambda _v q(\alpha   \phi  (r-q)-(\alpha ^2+1) (p \lambda _u+3  \lambda _1 ))-\left(\alpha ^2+1\right) q^2 \lambda _v^2=0,}
\end{array}\end{equation}
\begin{equation}\label{e:n28}\begin{array}{lcl}
(-27   -12 r -p r -q r -r^2)r \phi\lambda _1^2 +(81 +36 r +3 r^2)
\alpha \lambda _1^3-18 p r \phi \lambda _1 \lambda _u\\+81 p \alpha
\lambda _1^2 \lambda _u+24 p r \alpha \lambda _1^2 \lambda _u+(-3
p^2 r \phi -3 p r^2 \phi +27 p^2 \alpha \lambda _1 +4 p^2 r \alpha
\lambda _1 \\+4 p r^2 \alpha \lambda _1) \lambda _u^2+3 p^3 \alpha
\lambda _u^3-3 p r^2 \alpha \lambda _u^3+(-18 q r \phi \lambda _1+81
q \alpha \lambda _1^2+24 q r \alpha  \lambda _1^2\\-6 p q r \phi
\lambda _u+54 p q \alpha \lambda _1 \lambda _u+8 p q r \alpha
\lambda _1 \lambda _u+9 p^2 q \alpha \lambda _u^2) \lambda _v+(-3
q^2 r \phi -3 q r^2 \phi \\+27 q^2 \alpha  \lambda _1+4 q^2 r \alpha
\lambda _1+4 q r^2 \alpha  \lambda _1+9 p q^2 \alpha  \lambda _u)
\lambda _v^2+\left(3 q^3 \alpha -3 q r^2 \alpha \right) \lambda
_v^3=0,
\end{array}\end{equation}
respectively.

Rewrite (\ref{e:n26}), (\ref{e:n27}) and (\ref{e:n28}) as
polynomials
${\mathcal{F}_{2}}_{(\lambda_1,\lambda_u,\alpha,\phi)}(\lambda_v),
{G_{2}}_{(\lambda_1,\lambda_u,\alpha,\phi)}(\lambda_v)$,
 and ${\mathcal{G}_1}_{(\lambda_1,\lambda_u,\alpha,\phi)}(\lambda_v)$ of
$\lambda_v$ with coefficients in polynomial ring
$\mathbb{R}[\lambda_1,\lambda_u,\alpha,\phi]$. According to
Lemma 2.1, the equations
${\mathcal{F}_{2}}_{(\lambda_1,\lambda_u,\alpha,\phi)}(\lambda_v)=0,
{G_{2}}_{(\lambda_1,\lambda_u,\alpha,\phi)}(\lambda_v)=0$, and
${\mathcal{F}_{2}}_{(\lambda_1,\lambda_u,\alpha,\phi)}(\lambda_v)=0,
{\mathcal{G}_1}_{(\lambda_1,\lambda_u,\alpha,\phi)}(\lambda_v)=0$
have a common root if and only if
\begin{equation}\label{e:n29}\begin{array}{lcl}
\mathfrak{R}({\mathcal{F}_{2}}_{(\lambda_1,\lambda_u,\alpha,\phi)}(\lambda_v),
{G_{2}}_{(\lambda_1,\lambda_u,\alpha,\phi)}(\lambda_v))=0,
\end{array}\end{equation}
and
\begin{equation}\label{e:n30}\begin{array}{lcl}
\mathfrak{R}({\mathcal{F}_{2}}_{(\lambda_1,\lambda_u,\alpha,\phi)}(\lambda_v),
{\mathcal{G}_1}_{(\lambda_1,\lambda_u,\alpha,\phi)}(\lambda_v))=0,
\end{array}\end{equation}
which is  polynomials of $\lambda_1,\lambda_u,\alpha,\phi$  with
real coefficients, in degree $4$ and $6$ respectively. Rewrite,
equations (\ref{e:n29}) and (\ref{e:n30}) as polynomials
${\mathcal{F}_3}_{(\lambda_1,\alpha,\phi)}(\lambda_u),
{\mathcal{G}_2}_{(\lambda_1,\alpha,\phi)}(\lambda_u)$ with
coefficients in polynomial ring $\mathbb{R}[\lambda_1,\alpha,\phi]$ over real
field. According to Lemma 2.1, the equations
${\mathcal{F}_3}_{(\lambda_1,\alpha,\phi)}(\lambda_u)=0,
{\mathcal{G}_2}_{(\lambda_1,\alpha,\phi)}(\lambda_u)=0$, have a
common root if and only if
\begin{equation}\label{e:n31}\begin{array}{lcl}
\mathfrak{R}({\mathcal{F}_{3}}_{(\lambda_1,\alpha,\phi)}(\lambda_u),
{\mathcal{G}_2}_{(\lambda_1,\alpha,\phi)}(\lambda_u))=0,
\end{array}\end{equation}
which is a polynomial equation
\begin{equation}\label{e:d61}\begin{array}{rcl}
{\mathcal{F}_{4}}(\alpha,\phi,\lambda_1)=0.
\end{array}
 \end{equation}

Differentiating (\ref{e:d61}) with respect to $e_1$ and using
(\ref{e:d44}) and (\ref{e:d59}), we find a polynomial
\begin{equation}\label{e:n32}\begin{array}{rcl}
{\mathcal{F}_{5}}(\alpha,\phi,\lambda_1)=0.
\end{array}
 \end{equation}

Differentiating (\ref{e:n32}) with respect to $e_1$ and using
(\ref{e:d44}) and (\ref{e:d59}), we find a polynomial
\begin{equation}\label{e:n33}\begin{array}{rcl}
{\mathcal{F}_{6}}(\alpha,\phi,\lambda_1)=0.
\end{array}
 \end{equation}

 Rewrite (\ref{e:d61}), (\ref{e:n32}) and (\ref{e:n33}) as
polynomials ${\mathcal{F}_{4}}_{(\lambda_1,\phi)}(\alpha),
{\mathcal{F}_{5}}_{(\lambda_1,\phi)}(\alpha)$,
 and ${\mathcal{F}_{6}}_{(\lambda_1,\phi)}(\alpha)$ of
$\alpha$ with coefficients in polynomial ring $\mathbb{R}[\lambda_1,\phi]$.
According to Lemma 2.1, the equations
${\mathcal{F}_{4}}_{(\lambda_1,\phi)}(\alpha)=0,
{\mathcal{F}_{5}}_{(\lambda_1,\phi)}(\alpha)=0$, and
${\mathcal{F}_{4}}_{(\lambda_1,\phi)}(\alpha)=0,
{\mathcal{F}_{6}}_{(\lambda_1,\phi)}(\alpha)=0$ have a common root
if and only if
\begin{equation}\label{e:n34}\begin{array}{lcl}
\mathfrak{R}({\mathcal{F}_{4}}_{(\lambda_1,\phi)}(\alpha),
{\mathcal{F}_{5}}_{(\lambda_1,\phi)}(\alpha))=0,
\mathfrak{R}({\mathcal{F}_{4}}_{(\lambda_1,\phi)}(\alpha),
{\mathcal{F}_{6}}_{(\lambda_1,\phi)}(\alpha))=0,
\end{array}\end{equation}
which are  polynomial equations
\begin{equation}\label{e:n35}\begin{array}{rcl}
h_1(\phi, \lambda_1)=0, \quad\mbox{and} \quad h_2(\phi,
\lambda_1)=0,
\end{array}
 \end{equation}
respectively. Finally, rewrite ${h_1}_{\lambda_1}(\phi)=0$ and
${h_2}_{\lambda_1}(\phi)=0$,  as a polynomial equations $\phi$ with
coefficients in polynomial ring $\mathbb{R}[\lambda_1]$ over real field.
According to Lemma 2.1, the equations
${h_1}_{\lambda_1}(\phi)=0$ and ${h_2}_{\lambda_1}(\phi)=0$ have a
common root if and only if
\begin{equation}\label{e:n36}\begin{array}{lcl}
\mathfrak{R}({h_1}_{\lambda_1}(\phi), {h_2}_{\lambda_1}(\phi))=0,
\end{array}\end{equation}
which is a  polynomial equation ${h_3}(\lambda_1)=0$ of $\lambda_1$
with constant coefficients. Thus, the real function $\lambda_1$
satisfies a polynomial equation ${h_3}(\lambda_1)=0$ with constant
coefficients, and, therefore, $\lambda_1$ must be a constant, a contradiction. 

\medskip
\noindent
\textbf{Case 2.} \emph{The case of three distinct principal curvatures}

Suppose that $M$ is a biconservative hypersurface with three
distinct principal curvatures $\lambda_1=-\frac{nH}{2}, \lambda_u$,
$\lambda_v$, with multiplicities $1, p$ and $n-p-1$ respectively.
Further, suppose that $M$ has constant norm of second fundamental
forms. Without losing generality, we choose $e_{1}$ in the direction
of $\grad H$ and therefore shape operator $\mathcal{A}$ of the
hypersurface will take the following form with respect to a suitable
frame  $\{e_{1}, e_{2}, \dots,e_{n}\}$
\begin{equation}\label{e:d62}
\mathcal{A}e_1=-\frac{nH}{2}e_1,  \quad
\mathcal{A}e_i=\lambda_u e_i,\quad \mathcal{A}e_j=\lambda_v
e_j,
\end{equation}
where $i=2,3,\dots, p+1$, and $j=p+2, p+3,\dots, n.$

 Using (\ref{e:d62}) in (\ref{e:c11}) and (\ref{e:g1}), we get
\begin{equation}\label{e:d63}
p\lambda_u+(n-p-1)\lambda_{v}=-3\lambda_1.
\end{equation}
\begin{equation}\label{e:d64}
p\lambda_u^{2}+(n-p-1)\lambda_{v}^{2}=\beta-\lambda_1^2.
\end{equation}

Eliminating $\lambda_v$ from (\ref{e:d64}) using (\ref{e:d63}), we
obtain
\begin{equation}\label{e:d65}
8 \lambda _1^2+n \lambda _1^2-p \lambda _1^2+6 p \lambda _1 \lambda
_u-p \lambda _u^2+n p \lambda _u^2=(-1+n-p) \beta.
\end{equation}

Similarly, eliminating $\lambda_u$ from (\ref{e:d64}) using
(\ref{e:d63}), we get
\begin{equation}\label{e:d66}
9 \lambda _1^2+p \lambda _1^2+\left(-6 \lambda _1+6 n \lambda _1-6 p
\lambda _1\right) \lambda _v+\left(1-2 n+n^2+p-n p\right) \lambda
_v^2=p \beta.
\end{equation}

Differentiating (\ref{e:d65}) with respect to $e_1$, we find
\begin{equation}\label{e:d70}
e_1(\lambda_u) =\mu e_1(\lambda_1),\quad e_1(\mu)=\frac{-8+p
\left(1-6 \mu +\mu ^2\right)-n \left(1+p \mu ^2\right)}{p \left(3
\lambda _1+(-1+n) \lambda _u\right)}e_1(\lambda_1),
\end{equation}
where $\mu=-\frac{(8+n-p) \lambda _1+3 p \lambda _u}{p \left(3
\lambda _1+(-1+n) \lambda _u\right)}$ and to find $e_1(\mu)$ we have
used the first expression of (\ref{e:d70}).

Differentiating (\ref{e:d65}) with respect to $e_j$ and using
(\ref{e:c6}), we find
\begin{equation}\label{e:d67}
e_{j}(\lambda_u)=0, \omega^j_{ii}=0,
\end{equation}
where $i=2,3,\dots, p+1$, and $j=p+2, p+3,\dots, n.$

Differentiating (\ref{e:d66}) with respect to $e_i$ and using
(\ref{e:c6}), we have
\begin{equation}\label{e:d68}
e_{i}(\lambda_v)=0, \omega^i_{jj}=0,
\end{equation}
where $i=2,3,\dots, p+1$, and $j=p+2, p+3,\dots, n.$

Also, putting  $k\in \{p+2, p+3,\dots, n\}$, and $j, i \in
\{2,3,\dots, p+1\}$ in (\ref{e:c7}) and using (\ref{e:c5}) and
(\ref{e:c7}), we get
\begin{equation}\label{e:d69}
\omega^j_{ik}=\omega^k_{ij}=\omega^k_{ji}=\omega^i_{jk}=0, \quad
i\neq j.
\end{equation}

 Evaluating $g(R(e_i,e_j)e_i,e_j)$ using (\ref{e:b5}),
(\ref{e:d62}), (\ref{e:d67}), (\ref{e:d68}) and (\ref{e:d69}), we
obtain
\begin{equation}\label{e:d71}
\omega_{ii}^1\omega_{jj}^1=-c-\lambda_u\lambda_v,\quad\mbox{for}\quad
j\in \{p+2, p+3,\dots, n\}, i \in \{2,3,\dots, p+1\},
\end{equation}
where $\omega_{ii}^1=\frac{e_{1}(\lambda_u)}{\lambda_u-\lambda_1},
\omega_{jj}^1=\frac{e_{1}(\lambda_v)}{\lambda_v-\lambda_1}$.

Also, equation (\ref{e:c14}) is valid for three distinct principal
curvatures. Therefore, from (\ref{e:c14}), we find
\begin{equation}\label{e:d72}
e_{1}(\frac{e_{1}(\lambda_u)}{\lambda_u-\lambda_1})=(\frac{e_{1}(\lambda_u)}{\lambda_u-\lambda_1})^{2}+\lambda_1\lambda_u+c,
\end{equation}
\begin{equation}\label{e:d73}
e_{1}(\frac{e_{1}(\lambda_v)}{\lambda_v-\lambda_1})=(\frac{e_{1}(\lambda_v)}{\lambda_v-\lambda_1})^{2}+\lambda_1\lambda_v+c.
\end{equation}

Using (\ref{e:c6}), (\ref{e:d63}) and (\ref{e:d70}) in
(\ref{e:d71}), we get
\begin{equation}\label{e:d74}\begin{array}{lcl}
 \frac{\lambda _u \left(-\lambda _1+\lambda _u\right) \left(3 \lambda _1+p \lambda _u\right) \left((2+n-p) \lambda _1
 +p \lambda _u\right)}{-1+n-p}=e_1^2(\lambda_1) \mu  (3+p \mu ).\end{array}
\end{equation}

Using (\ref{e:c6}) and (\ref{e:d70}) in (\ref{e:d72}), we obtain
\begin{equation}\label{e:d75}\begin{array}{lcl}
 \mu(\lambda_u-\lambda_1)e_1e_1(\lambda_1)= \lambda _1 \lambda _u \left(-\lambda _1
 +\lambda _u\right){}^2+A e_1^2(\lambda_1),\end{array}
\end{equation}
where $A=\frac{-\left(8+n+n p \mu ^2-p \left(1-6 \mu +\mu
^2\right)\right) \left(\lambda _1-\lambda _u\right)+p \mu  (-1+2 \mu
) \left(3 \lambda _1+(-1+n) \lambda _u\right)}{p \left(3 \lambda
_1+(-1+n) \lambda _u\right)}.$

 Using (\ref{e:c6}) and (\ref{e:d70}) in (\ref{e:d73}), we find
\begin{equation}\label{e:d76}\begin{array}{rcl}
 (3+2\mu)(2\lambda+4\lambda_1)e_1e_1(\lambda_1)=Be_1^2(\lambda_1)-\lambda_1(3\lambda_1+2\lambda)(2\lambda+4\lambda_1)^2,\end{array}
\end{equation}
where
$B=\frac{[\lambda(36\mu^2+102\mu+83)+\lambda_1(48\mu^2+126\mu+103)]}{3(\lambda_1+\lambda)}.$

Eliminating $e_1e_1(\lambda_1)$ from (\ref{e:d75}) using
(\ref{e:d76}), we obtain
\begin{equation}\label{e:d77}\begin{array}{rcl}
 C e_1^2(\lambda_1)=D,\end{array}
\end{equation}
where $C=(3+2\mu)(2\lambda+4\lambda_1)A-\mu(\lambda-\lambda_1)B$ and $
D=-\lambda_1(\lambda-\lambda_1)(2\lambda+4\lambda_1)[\mu(3\lambda_1+2\lambda)(4\lambda_1+2\lambda)+\lambda(3+2\mu)(\lambda-\lambda_1)].$

Eliminating $e_1^2(\lambda_1)$ from (\ref{e:d74}) using
(\ref{e:d77}), we find
\begin{equation}\label{e:d78}\begin{array}{lcl}
 (3\mu+2\mu^2)D-C
\lambda(\lambda-\lambda_1)(3\lambda_1+2\lambda)(4\lambda_1+2\lambda)=0.\end{array}
\end{equation}

Finally, eliminating $\lambda$ from (\ref{e:d65}) and (\ref{e:d78}),
we get a polynomial equation $\varphi(H)=0$ in $H$ with constant
coefficients. Thus, the real function $H$ satisfies a polynomial
equation $\varphi(H) = 0$ with constant coefficients, and,
therefore, it must be a constant.

\medskip
\noindent
\textbf{Case 3.} \emph{The case of two distinct principal curvatures}

Suppose that $M$ is a biconservative hypersurface with two
distinct principal curvatures $\lambda_1=-\frac{nH}{2}$ and $\lambda$, with multiplicities $1$ and $n-1$ respectively. Without losing generality, we choose $e_{1}$ in the direction
of $\grad H$ and therefore shape operator $\mathcal{A}$ of the
hypersurface will take the following form with respect to a suitable
frame  $\{e_{1}, e_{2}, \dots,e_{n}\}$
\begin{equation}\label{e:d79} \mathcal{A}e_1=-\frac{nH}{2}e_1,  \quad
\mathcal{A}e_i=\lambda e_i, \quad i=2,\dots,n-1.
\end{equation}

Using (\ref{e:d79}) in (\ref{e:c11}) and (\ref{e:g1}), we get
\begin{equation}\label{e:d80}
\lambda=-\frac{3\lambda_1}{n-1},
\end{equation}
\begin{equation}\label{e:d81}
(n-1)\lambda^{2}=\beta-\lambda_1^2,
\end{equation}
respectively.

Further, if $M$ has constant norm of second fundamental
forms, then, from (\ref{e:d80}) and (\ref{e:d81}), we get $\lambda_1$ is a
constant, which gives that $H$ constant, a contradiction.

\noindent
Combining \textbf{Cases 1, 2}, and \textbf{3} it follows that $M$ has constant mean curvature. 

\medskip
Now,  using (\ref{e:b5}) we get that the scalar curvature  $\rho$  is given by
\begin{equation}\label{e:d82}
\rho = n(n-1)c+n^2H^2+\beta,
\end{equation}
which is also  constant.

% ------------------------------------------------------------------------
\bibliography{xbib}

\bigskip
\noindent
Author's address:\\
\textbf{Ram Shankar Gupta}\\
University School of Basic and Applied
Sciences, Guru Gobind Singh Indraprastha University, Sector-16C,
Dwarka, New Delhi-110078,
India\\
\textbf{Email:} ramshankar.gupta@gmail.com

\smallskip
\noindent
\textbf{Andreas Arvanitoyeorgos}\\
University of Patras, Department of Mathematics, GR-26504 Rion, Greece, and\\
Hellenic Open University, Aristotelous 18, GR-26335 Patras, Greece\\
\textbf{Email:} arvanito@math.upatras.gr
 \end{document}